\documentclass[12pt]{amsart}
\usepackage{amsfonts,amsmath,amssymb,amsthm}
\usepackage[right=3cm,left=3cm,top=3cm,bottom=3cm]{geometry}
\usepackage{enumerate}
\usepackage{xcolor}

\newtheorem{theorem}{Theorem}
\newtheorem{lemma}[theorem]{Lemma}
\newtheorem{proposition}[theorem]{Proposition}
\newtheorem{corollary}[theorem]{Corollary}

\theoremstyle{definition}
\newtheorem{definition}[theorem]{Definition}
\newtheorem{example}[theorem]{Example}

\newtheorem{notation}[theorem]{Notation}

\newtheorem{remark}[theorem]{Remark}

\title{The Frobenius problem for generalized repunit numerical semigroups}

\author{Manuel B. Branco}
\address{Departamento de Matem\'aticas, Universidade de \'Evora, 7000-671 \'Evora, Portugal}
\email{mbb@uevora.pt}
\author{Isabel Colaço}
\address{Departamento de Matem\'atica e Ci\^encias F\'{\i}sicas, Instituto Polit\'ecnico de Beja, 7800-295 Beja, Portugal}
\email{isabel.colaco@ipbeja.pt}
\author{Ignacio Ojeda}
\address{Departamento de Matem\'aticas, Universidad de Extremadura, 06071 Badajoz, Spain}
\email{ojedamc@unex.es}
\thanks{This work was supported in part by Grant MTM2017-84890-P funded by MCIN/FEDER-UE, by Grant PGC2018-096446-B-C21 funded by MCIN/AEI/10.13039/501100011033, by "ERDF A way of making Europe",  and by Grant GR21055(FQM-024) funded by Junta de Extremadura (Spain)/FEDER funds.}

\begin{document}

\begin{abstract}
In this paper, we introduce and study the numerical semigroups generated by $\{a_1, a_2, \ldots \} \subset \mathbb{N}$ such that $a_1$ is the repunit number in base $b > 1$ of length $n > 1$ and $a_i - a_{i-1} = a\, b^{i-2},$ for every $i \geq 2$, where $a$ is a positive integer relatively prime with $a_1$. These numerical semigroups generalize the repunit numerical semigroups among many others. We show that they have interesting properties such as being homogeneous and Wilf. Moreover, we solve the Frobenius problem for this family, by giving a closed formula for the Frobenius number in terms of $a, b$ and $n$, and compute other usual invariants such as the Ap\'ery sets, the genus or the type.
\end{abstract}

\maketitle

\section{Introduction}\label{Sect1}

Let $\mathbb{N}$ be the set of non-negative integers. A numerical semigroup $S$ is a subset of $\mathbb{N}$  containing zero which is closed under addition of natural numbers and such that $\mathbb{N} \setminus S$ is finite. The cardinality of $\mathbb{N} \setminus S$ is called the genus of $S$, denoted $\operatorname{g}(S)$.

Numerical semigroups have a unique finite minimal system of generators, that is, given a numerical semigroup $S$ there exists a unique set $\{a_1, \ldots, a_e\} \subset \mathbb{N}$ such that \[S = \mathbb{N} a_1 + \ldots + \mathbb{N} a_e\] and no proper subset of $\{a_1, \ldots, a_e\}$ generates $S$ (see \cite[Theorem 2.7]{libro}). In this case, the set $\{a_1, \ldots, a_e\}$ is the minimal system of generators of $S$, its cardinality is called the embedding dimension of $S$, denoted $\operatorname{e}(S)$, and $\min \{a_1, \ldots, a_e\}$ is called the multiplicity of $S$, denoted $\operatorname{m}(S)$. Notice that the finiteness of the genus implies that $\gcd(a_1, \ldots, a_e) = 1$. In fact, one has that the necessary and sufficient condition for a subset $\mathcal{A}$ of $\mathbb{N}$ to generate a numerical semigroup is $\gcd(\mathcal{A}) = 1$ (see, e.g., \cite[Lemma 2.1]{libro}).

Let $S$ be a numerical semigroup. Since $\mathbb{N} \setminus S$ is finite, there exists the greatest integer not in $S$ which is called the Frobenius number of $S$, denoted $\operatorname{F}(S)$. The so-called Frobenius problem deals with finding a closed formula for $\operatorname{F}(S)$ in terms of the minimal systems of generators of $S$, if possible (see, e.g. \cite{ramirez-alfonsin}).

Let $b$ be a positive integer greater than $1$. Set $a_1 = \sum_{j=0}^{n-1} b^j$ and consider the sequence $(a_i)_{i \geq 1}$ defined by the recurrence relation \[a_i - a_{i-1} = a\, b^{i-2},\ \text{for every}\ i \geq 2,\] where $a$ and $n$ are positive integers. In this paper, we study the numerical semigroups, $S_a(b,n)$, generated by $\{a_1, a_2, \ldots \},$ provided that $\gcd(a_1, a) = 1$. This last condition is necessary and sufficient for $S_a(b,n)$ to be a numerical semigroup (see Proposition \ref{prop4}). In this case, we say that $S_a(b,n)$ is a generalized repunit numerical semigroup as it generalizes the repunit numerical semigroups studied in \cite{RBT16} (see Example \ref{exa7}).

Clearly, the generalized repunit numerical semigroup $S_a(b,n)$ has multiplicity $a_1$. Moreover, by Theorem \ref{teo6}, we have that $S_a(b,n)$ is minimally generated by $\{a_1, \ldots, a_n\}$; in particular, the embedding dimension of $S_a(b,n)$ is $n$.

The main results in this paper are Theorem \ref{teo22}, which provides the following formula for the Frobenius number of $S_a(b,n)$:
\[
\operatorname{F}(S_a(b,n)) = \left\{
\begin{array}{ccc}
(n-1)\,\left(b^n-1-a\right)+a\, \left(\sum_{j=0}^{n-1} b^j\right) & \text{if} & a < b^n-1; \\
b^n-1-a + a\, \left(\sum_{j=0}^{n-1} b^j\right) & \text{if} & a > b^n-1,
\end{array}\right. 
\]
and Corollary \ref{cor26}, which gives the following formula for the genus of $S_a(b,n)$:
\[
\operatorname{g}(S_a(b,n)) = \frac{(n-1)\, b^n + a\, \left(\sum_{j=1}^{n-1} b^j\right)}2.
\]
To achieve these results we take advantage of Selmer's formulas, summarized in Proposition \ref{prop20}. These formulas depend on the Ap\'ery sets of $S_a(b,n)$. We explicitly compute the Ap\'ery set of $S_a(b,n)$ with respect to $a_1$ (Theorem \ref{teo15}). This is a result that may seem technical, however it reflects the internal structure of generalized repunit numerical semigroups. For instance, the Ap\'ery set of $S_a(b,i)$ can be obtained from the Ap\'ery set of $S_a(b,i-1)$, for every $i \geq 3$. This is Corollary \ref{cor19} whose statement is a stronger version of \cite[Theorem 3.3]{OP18} partially thanks to the fact that generalized repunit numerical semigroups are homogeneous in the sense of \cite{JZ18} (Proposition \ref{prop17}).

The last section of the paper is devoted to the computation of the pseudo-Frobenius numbers of $S_a(b,n)$. Concretely, using our results in Section \ref{Sect3}, we explicitly compute the whole set of pseudo-Frobenius numbers of $S_a(b,n)$ and we obtain that its cardinality is $n-1$ (Proposition \ref{prop29}). So, we prove that the type of $S_a(b,n)$ is equal to $n-1$ which implies that $S_a(b,n)$ is Wilf (see Section \ref{Sect5} for further details).

Generalized repunit numerical semigroups have other interesting properties. Without going further, using Gr\"obner basis techniques, in \cite{BCO2} it is proved that the toric ideal associated to $S_a(b,n)$ is determinantal and that the cardinal of any minimal presentation of $S_a(b,n)$ is $\binom{n}2$. Moreover, following \cite{OjVi2}, we proved that, for $n > 3$, generalized repunit numerical semigroups are uniquely presented, in the sense of \cite{GSOj}, if and only if $a < b-1$.

Finally, we note that our results are also valid for the case $b = 1$. In this case, $a_i = n+(i-1)\, a,\ i \geq 1$, is an arithmetic sequence that generates a MED semigroup, provided that $\gcd(n,a)=1$. These semigroups are widely known (see, e.g. \cite [Section 3]{libro}); for this reason and for the sake of simplicity, we consider $b > 1$; so that $a_1$ is properly a repunit number.

\section{Generalized repunit numerical semigroups}\label{Sect2}

Let $b > 1$ be a positive integer. 

\begin{definition}\label{def1}
A repunit number in base $b$ is an integer whose representation in base $b$ contains only the digit $1$.
\end{definition}

We write $r_b(\ell)$ for the repunit number in base $b$ of length $\ell$, that is, \[r_b(\ell) = \sum_{j=0}^{\ell-1} b^j = \frac{b^\ell-1}{b-1}.\] By convention, we assume $r_b(0) = 0$.

\begin{example}
The first six repunit numbers in base $2$ are $1, 3, 7, 15, 31, 63 \ldots$, whereas the first six repunit numbers in base $3$, are $1, 4, 13, 40, 121, 364 \ldots$. Observe that repunit numbers in base $2$ are the Mersenne numbers.
\end{example}

Here and in what follows, $a$ and $n$ denote two positive integers. 

\begin{notation}\label{not3}
Set $a_i := r_b(n) + a\, r_b(i-1),\ i \geq 1$. Observe that $a_1 = r_b(n)$ and $a_i - a_{i-1} = a\, b^{i-2},$ for every $i \geq 2$. We write $S_a(b,n)$ for the submonoid of $\mathbb{N}$ generated by $a_i,\ i \geq 1$. 
\end{notation}

If $n=1,$ then $a_1 = 1$ and therefore $S_a(b,n) = \mathbb{N}$. So, in the following we assume that $n > 1$.

\begin{proposition}\label{prop4}
$S_a(b,n)$ is a numerical semigroup if and only if $\gcd(r_b(n),a) = 1$.
\end{proposition}

\begin{proof}
Let $d = \gcd(r_b(n),a)$. By definition, $S_a(b,n) \subseteq d\, \mathbb{N}$. Now, if  $S_a(b,n)$ is a numerical semigroup, then $\mathbb{N} \setminus d\, \mathbb{N} \subset \mathbb{N} \setminus S_a(b,n)$ has finitely many elements, and hence $d = 1$. Conversely, if $\gcd(r_b(n),a) = 1$, then $\gcd(a_1, a_2) = \gcd(r_b(n), r_b(n)+a) = \gcd(r_b(n),a) = 1$. So, $a_1 \mathbb{N} + a_2 \mathbb{N}$ is a numerical semigroup containing $S_a(b,n)$. Therefore $\mathbb{N} \setminus S_a(b,n) \subseteq \mathbb{N} \setminus (a_1 \mathbb{N} + a_2 \mathbb{N})$  has finitely many elements, that is to say, $S_a(b,n)$ is a numerical semigroup. 
\end{proof}

Let us prove that if $\gcd(r_b(n),a) = 1$, then $\{a_1, \ldots, a_n\}$ is the minimal generating set of $S_a(b,n)$. We begin with a useful lemma.  

\begin{lemma}\label{lema5}
The following equality holds: $a_{n+i} = a_i + a\, b^{i-1}\, a_1,$ for all $i \geq 1$.
\end{lemma}

\begin{proof}
It suffices to observe that $r_b(n+i-1) = r_b(i-1) + b^{i-1}\, r_b(n)$, for all $i \geq 1$, and, consequently, that $a_{n+i} = a_1 + a\, r_b(n+i-1) = a_1 + a\, ( r_b(i-1) + b^{i-1} r_b(n)) = a_i + a\, b^{i-1}\, a_1$, for all $i \geq 1$.
\end{proof}

Observe that the previous lemma already implies that $\{a_1, \ldots, a_n\}$ generates $S_a(b,n)$. So, it remains to see that $\{a_1, \ldots, a_n\}$ is minimal for the inclusion.
In this case, by \cite[Theorem 2.7]{libro}, $\{a_1, \ldots, a_n\}$ will be the (unique) minimal system of generators of $S_a(b,n)$.

\begin{theorem}\label{teo6}
If $S_a(b,n)$ is a numerical semigroup, then $\{a_1, \ldots, a_n\}$ is the minimal system of generators of $S_a(b,n)$. In particular, the embedding dimension of $S_a(b,n)$ is $n$.
\end{theorem}

\begin{proof}
By Lemma \ref{lema5}, we have that $\{a_1, \ldots, a_n\}$ is a system of generators of $S_b(a,n)$. Now, since $a_1 < \ldots < a_n$, to see the minimality property, it suffices to prove that $a_i \not\in \langle a_1, \ldots, a_{i-1}\rangle$ for every $i \in \{2, \ldots, n\}$. By the condition $\gcd(a_1, a_2) = \gcd(a_1, a) = 1$ this is true for $i=2$. Also when $a=1$, we have that $a_i - a_k = r_b(i-1)-r_b(k-1) = \sum_{j=k-1}^{i-2} b^j < a_1,$ for every $k \leq i$, and consequently, $a_i \not\in \langle a_1, \ldots, a_{i-1}\rangle$. 

So, from now on we assume $a > 1$ and $i \in \{3, \ldots, n\}$. If $a_i \in \langle a_1, \ldots, a_{i-1}\rangle$, then there exist $u_1, \ldots, u_{i-1} \in \mathbb{N}$ such that $a_i = \sum_{j=1}^{i-1} u_j a_j$. Therefore, $a_i = a_1 + a\, r_b(i-1)$ is equal $\big(\sum_{j=1}^{i-1} u_j\big) a_1 + a\, \sum_{j=1}^{i-1} u_j r_b(j-1)$ and thus \[\left(\sum_{j=1}^{i-1} u_j\right) a_1 \equiv a_1\ (\!\!\!\!\!\mod a ).\] 

Now, since $S_a(b,n)$ is a numerical semigroup, by Proposition \ref{prop4}, we have $\gcd(a_1,a) = \gcd(r_b(n),a) = 1$, and we conclude that $\sum_{j=1}^{i-1} u_j \equiv 1\ (\!\!\! \mod a)$. If $\sum_{j=1}^{i-1} u_j  = 1$, then there exists $k \in \{1, \ldots, i-1\}$ such that $u_k = 1$ and $u_j = 0$ for every $j \neq k$, that is to say, $a_i = a_k$ which is not possible because $k < i$. Thus, there exists a positive integer $N$ such that $\sum_{j=1}^{i-1} u_j = 1 + N\, a$. Therefore, $a_i = (1+N\,a)\, a_1 \geq (1+a)\, a_1 = a_{n+1}$, where the last equality follows from Lemma \ref{lema5}. However, this inequality implies $i \geq n+1$, in contradiction to our assumption.
\end{proof}

We emphasize that the hypothesis $S_a(b,n)$ is a numerical semigroup (equivalently, $\gcd(a_1, \ldots, a_n)=1$) cannot be avoided for the minimality property of $\{a_1, \ldots, a_n\}$; for example, if $b=2, a=5$ and $n = 4,$ we have that $a_1 = 15, a_2 = 20, a_3 = 30$ and $a_4 = 50$, clearly, $a_1$ and $a_2$ suffice to generate $S_a(b,n)$, in this case.

Here and throughout this section, we suppose $\gcd(r_b(n),a) = 1$ so that $S_a(b,n)$ is a numerical semigroup with multiplicity $a_1$ and, by Theorem \ref{teo6}, of embedding dimension $n$. We call these semigroups generalized repunit numerical semigroups or grepunit semigroups for short. 

\begin{example}\label{exa7}\mbox{}
\begin{enumerate}[i)]
\item If $a = 2^n, a_1 = 2^n-1$ and $b = 2$, then $S_a(b,n)$ is a Merssene numerical semigroup (see \cite{RBT17}).
\item If $a = b^n$, then $S_a(b,n)$ is a repunit numerical semigroup (see \cite{RBT16}).
\end{enumerate}
\end{example}

The numerical semigroups in Example \ref{exa7} are part of the larger family of those numerical semigroups which are closed respect to the action of affine maps. A numerical semigroup $S$ is said to be closed respect to the action of an affine map if there exists $\alpha \in \mathbb{N} \setminus \{0\}$ and $\beta \in \mathbb{Z}$ such that $\alpha s + \beta \in S,$ for every $s \in S \setminus \{0\}$.

\begin{corollary}\label{cor8}
$S_a(b,n)$ is closed by the action of the affine map $x \mapsto b\, x + a-(b^n-1)$.
\end{corollary}

\begin{proof}
We first, observe that \begin{align*} 
b\, a_j + a-(b^n-1) & = b\, a_j + a - (b-1) r_b(n) \\
&  = b\, r_b(n) - (b-1)\, r_b(n) + a\, b\, r_b(j-1) + a \\ & = r_b(n) + a\, b\, r_b(j-1) + a = r_b(n) + a\, r_b(j) \\ & = a_{j+1},\end{align*} for every $j \geq 1$. 

Now, since $\{a_1, a_2, \ldots, a_n\}$ generates $S_a(b,n)$, given $s \in S \setminus \{0\}$, there exist $u_i \in \mathbb{N},\ i = 1, \ldots, n,$ with $u_j \neq 0$ for some $j$, such that $s = \sum_{i=1}^n u_i a_i$. Therefore
\begin{align*}
b\, s + a-(b^n-1) & = \sum_{i=1}^n (u_i b)\, a_i + a-(b^n-1) \\
& = \sum_{\substack{i=1 \\ i \neq j}}^{n} (u_i b)\, a_i + (u_j b)\, a_j + a-(b^n-1) \\ 
& = \sum_{\substack{i=1 \\ i \neq j}}^{n} (u_i b)\, a_i + \big((u_j-1) b\big)\, a_j +  b\, a_j + a-(b^n-1) \\ & = \sum_{\substack{i=1 \\ i \neq j}}^{n} (u_i b)\, a_i + \big((u_j-1) b\big)\, a_j +  a_{j+1} \in S,
\end{align*}
as claimed.
\end{proof}

In \cite{Ugo17}, numerical semigroups which are closed respect to the action of the affine maps $x \mapsto \alpha x + \beta$, with $\alpha \in \mathbb{N} \setminus \{0\}$ and $\beta \in \mathbb{N}$, are studied. Therefore, by Corollary \ref{cor8}, the grepunit semigroup $S_a(b,n)$ belongs to the family studied in \cite{Ugo17} if and only if $a - (b^n-1) > 0$; equivalently, $a > b^n-1$.

\begin{remark}\label{rem9}
Grepunit semigroups could be seen also as shifted numerical monoids in the sense of \cite{OP18}; since, by Theorem \ref{teo6}, $S_a(b,n)$ is minimally generated by \[ \{a_1, a_1 +  a\, r_b(1), \ldots, a_1 +  a\, r_b(n-1)\}.\] Nevertheless, the hypothesis $r_b(n) = a_1 > a^2\, r_b(n-1)^2$ required in \cite{OP18} only holds for grepunit semigroups when $n=2$ and $a^2 < b+1$. Indeed, $r_b(n) > a^2\, r_b(n-1)^2$ if and only if \[a^2 < \frac{r_b(n)}{r_b(n-1)^2} = \frac{1+b\, r_b(n-1)}{r_b(n-1)^2} = \frac{1}{r_b(n-1)^2} + \frac{b}{r_b(n-1)}.\] Now, as the right hand side is either $1+b$, if $n=2$, or less than $(1+2b)/(1+b)^2 < 1$, otherwise, we are done.
\end{remark}

To finish this section, we make explicit a set of relations which characterizes $S_a(b,n)$.

\begin{lemma}\label{lema10}
For each pair of integers $i \geq 1$ and $j \geq 1$, it holds that $b\, a_i + a_{i+j} = b\, a_{i+j-1} + a_{i+1}$.
\end{lemma}

\begin{proof}
Since $a_{i+j} = a_1 + a\, r_b(i+j-1) = a_1 + a\, (r_b(i-1) +  b^{i-1}\, r_b(j)) = a_i + a\, b^{i-1}\, r_b(j)$, for every $j$, we have that \begin{align*}
b\, a_i + a_{i+j} & = b\, a_i  + a_i + a\, b^{i-1}\, r_b(j) \\ & = b\, a_i  + a_i + a\, b^ {i-1}\, (b\, r_b(j-1)+1) \\ & = b\, (a_i + a\, b^{i-1}\, r_b(j-1)) + a_i + a\, b^{i-1} \\ & =  b\, a_{i+j-1} + (a_1 + a\, r_b(i-1)) + a\, b^{i-1} \\ & =  b\, a_{i+j-1} + a_{i+1},
\end{align*}
as claimed.
\end{proof}

Let us delve into what it is said in Lemma \ref{lema10}. Consider the subgroup $\mathcal{L}$ of $\mathbb{Z}^n$ generated by the rows of the $(n-1) \times n-$matrix 
\[ A = \left(\begin{array}{cccccccc} b & -(b+1) & 1 & 0 & \ldots & 0 & 0 & 0 \\ 
0 & b & -(b+1) & 1 & \ldots & 0 & 0 & 0  \\ 
\vdots & \vdots & \vdots & \vdots &  & \vdots & \vdots & \vdots\\ 
0 & 0 & 0 & 0 & \ldots & b & -(b+1) & 1 \\
(a+1) & 0 & 0 & 0 & \ldots & 0 & b & -(b+1)
\end{array}\right).\]
We first observe that, by Lemma \ref{lema5}, all the equalities in Lemma \ref{lema10} can be written in the form \[\mathbf{v}\, A \left(\begin{array}{c} a_1 \\ \vdots \\ a_n \end{array}\right) = 0\] for some $\mathbf{v} \in \mathbb{Z}^{n-1}$. Moreover, taking into account that, by Lemma \ref{lema5} again, $a_{n+1} = (a+1)\, a_1$, a direct computation shows that the maximal minors of $A$ are equal to $-a_1, a_2, -a_3, \ldots, (-1)^n a_n.$ So, $\mathbb{Z}^n/\mathcal{L}$ is a group of rank $n-1$ which is torsion free if and only if $\gcd(a_1, \ldots, a_n) =1$; equivalently, $S_a(b,n)$ is a numerical semigroup by Proposition \ref{prop4}. Thus, in this case, we have that the semigroup homomorphism \[S_a(b,n) \longrightarrow \mathbb{N}^n/\mathcal{L};\quad s = \sum_{i=1}^n u_i a_i \longmapsto (u_1, \ldots, u_n) + \mathcal{L},\] is an isomorphism; that is to say, $S_a(b,n)$ is the finitely generated commutative monoid corresponding to the congruence $\sim$ on $\mathbb{N}^n$ defined by $\mathbf{u} \sim \mathbf{v} \Longleftrightarrow \mathbf{u} - \mathbf{v} \in \mathcal{L}$.

Now, as a straightforward consequence of the results in \cite[Section 7.1]{MS05} we conclude the following.

\begin{corollary}\label{cor11}
Let $\Bbbk$ be a field. The semigroup ideal of $\Bbbk[x_1, \ldots, x_n]$ associated to $S_a(b,n)$ is equal to the ideal $I_\mathcal{L}$ generated by \begin{equation}\label{ecu1}
\{x_1^{u_1} \cdots x_n^{u_n} - x_1^{v_1} \cdots x_n^{v_n}\ \mid\ u_i-v_i \in  \mathcal{L},\ i = 1, \ldots, n\}.
\end{equation}
\end{corollary}

Observe that accordingly to the definition of $\mathcal{L}$ the binomials in $I_\mathcal{L}$ not involving $x_1$ are homogeneous.

\section{Ap\'ery sets of grepunit semigroups}\label{Sect3}

The main aim of this section is to determine the Apéry set of a grepunit semigroup with respect to its multiplicity. Let us start by recalling what Apéry sets of a numerical semigroup are.

\begin{definition}\label{def12}
Let $S$ be a numerical semigroup. The Ap\'ery set of $S$ with respect to $s \in S$, denoted $\operatorname{Ap}(S,s)$, is defined as
\[\operatorname{Ap}(S,s) = \{\omega \in S \mid \omega-s \not\in S\}.\]
For the sake of simplicity we write $\operatorname{Ap}(S)$ for the Ap\'ery set of $S$ with respect to its multiplicity, that is, $\operatorname{Ap}(S) = \operatorname{Ap}(S,\operatorname{m}(S))$.
\end{definition}

Let $a, b$ and $n$ be three positive integers such that $b > 1, n > 1$ and $\gcd(r_b(n), a) = 1$. As mentioned above, the main objective of this section is to compute $\operatorname{Ap}(S_a(b,n))$. To this end, we first introduce the sets $R(b,i)$.

\begin{definition}\label{def10}
Let $i \geq 2$ be an integer and define $R(b,i)$ to be the subset of $\mathbb{N}^{i-1}$ whose elements $(u_2, \ldots, u_i)$ satisfy 
\begin{enumerate}[(a)]
\item $0 \leq u_j \leq b,$ for every $j = 2, \ldots, i$;
\item if $u_j = b$, then $u_k = 0$ for every $k < j$.
\end{enumerate}
\end{definition}

Observe that 
\begin{equation}\label{ecu2}
R(b,i) = \left( R(b,i-1) \times \{0, \ldots, b-1\} \right) \cup \{(0, \ldots, 0, b)\} \subset \mathbb{N}^{i-1},
\end{equation}
for every $i \geq 3$.

\begin{lemma}\label{lema14}
The cardinality of $R(b,i)$ is equal to $r_b(i)$, for every $i \geq 2$.
\end{lemma}

\begin{proof}
We proceed by induction on $i$. If $i=2$, then $R(b,i) = \{0,1,\ldots,b\}$ and $r_b(i) = b+1$. Suppose that $i > 2$ and that the result is true for $i-1$. By \eqref{ecu2}, the cardinality of $R(b,i)$ is equal to $b$ times the cardinality of $R(b,i-1)$ plus one. Since, by induction hypothesis, the cardinality of $R(b,i-1)$ is equal to $r_b(i-1)$ and $r_b(i) = b\, r_b(i-1) + 1$, we are done.
\end{proof}

Recall that, by Proposition \ref{prop4}, we have that $S_a(b,n)$ is a grepunit semigroup if and only if $\gcd(r_b(n), a) = 1$. In this case, $S_a(b,n)$ is minimally generated by
\[a_1:=r_b(n), a_2 := r_b(n) + a\, r_b(1), \ldots, a_n :=  r_b(n) + a\, r_b(n-1), \] by Theorem \ref{teo6}. 
 
\begin{theorem}\label{teo15}
With the above notation, we have that \[\operatorname{Ap}(S_a(b,n)) = 
\left\{\sum_{i=2}^n u_i a_i \mid (u_2, \ldots, u_n) \in R(b,n)\right\}.\]
\end{theorem}

\begin{proof}
For the sake of simplicity of notation, we write $S$ for $S_a(b,n)$.

As $\operatorname{Ap}(S) \subset \mathbb{N}$, its elements are naturally ordered: $0 = \omega_1 < \ldots < \omega_{a_1}.$ So, we can proceed by induction on the index $j$ of $\omega_j$. If $j = 1$, then $\omega_j = 0$; so, by taking $(0, \ldots, 0) \in R(b,n)$ we are done. Suppose now that $j > 1$ and that the result is true for every $j' < j$. Let $k$ be the smallest index such that $\omega_j-a_k \in S$. Clearly, $(\omega_j - a_k) - a_1 \not\in S$; otherwise $\omega_j - a_1 = \big((\omega_j - a_k) - a_1\big) + a_k \in S$, in contradiction with the fact that $\omega_j \in \operatorname{Ap}(S)$. Therefore $\omega_j - a_k \in \operatorname{Ap}(S)$ and, by induction hypothesis, there exists $(u_2, \ldots, u_n) \in R(b,n)$ such that $\omega_j - a_k = \sum_{i=2}^n u_i a_i$. Thus, \[\omega_j = \sum_{i=2}^{k-1} u_i a_i + 
(u_k+1) a_k + \sum_{i=k+1}^{n} u_i a_i = (u_k+1) a_k + \sum_{i=l+1}^{n} u_i a_i,\] where the second equality follows from the minimality of $k$. Let us see that $(0, \ldots, 0, u_k+1, u_{k+1}, \ldots, u_n)$ lies in $R(b,n)$. If $u_k+1 \leq b$ and $u_i < b,\ i \in \{k+1, \ldots, n\}$, we are done. So, we distinguish two cases:
\begin{itemize}
\item If $u_k+1 > b$, then $u_k+1=b+1$. In this case, $(u_k + 1) a_k = (b+1)\, a_k = b\, a_k + a_k = b\, a_{k-1}+a_{k+1}$, where the last equality follows from Lemma \ref{lema10}.
\item If $u_i = b$, for some $i \in \{k+1, \ldots, n\}$, then $u_k = 0$ and so $(u_k +1) a_k + u_i a_i = a_k + b a_i = b\, a_{k-1} + a_{i+1}$, where the last equality follows from Lemma \ref{lema10} again.
\end{itemize}
In both cases, we obtain that $\omega_j - a_{k-1} \in S$ which contradicts the minimality of $k$. Hence none of these two cases can occur. 
\end{proof}

In \cite{JZ18} the notion of homogeneous numerical semigroups is introduced. 
Recall that if $S$ is a numerical semigroup minimally generated by $\{a_1, \ldots, a_n\}$, then the set of lengths of $s \in S$ is defined as
\[\mathsf{L}_S(s) := \left\{\sum_{j=1}^n u_j\ \mid\ s = \sum_{j=1}^n u_j\, a_j,\ u_j \geq 0 \right\}.\]

\begin{definition}\label{def16}
A numerical semigroup is said to be homogeneous if $\mathsf{L}_S(s)$ is a singleton for each $s \in \operatorname{Ap}(S)$.
\end{definition}

\begin{proposition}\label{prop17}
The numerical semigroup $S_a(b,n)$ is homogeneous.
\end{proposition}

\begin{proof}
By \cite[Proposition 3.9]{JZ18} and Corollary \ref{cor11}, it suffices to observe that one of the terms of each non-homogeneous element in \eqref{ecu1} for the standard grading on $\Bbbk[x_1,\ldots, x_n]$ is divisible by $x_1$.
\end{proof}

Let us see now that, for $i \in \{3, \ldots, n\}$, the Ap\'ery set $\operatorname{Ap}(S_a(b,i))$ can be constructed from the set $R(b,i-1)$. But, first we need a further piece of notation. 

Recall that, by Proposition \ref{prop4}, we have that $S_a(b,i)$ is a grepunit semigroup if and only if $\gcd(r_b(i), a) = 1$. In this case, $S_a(b,i)$ is minimally generated by \[a^{(i)}_1:=r_b(i), a^{(i)}_2 := a^{(i)}_1 + a\, r_b(1), \ldots, a^{(i)}_i :=  a^{(i)}_1 + a\, r_b(i-1), \] by Theorem \ref{teo6}. 

\begin{corollary}\label{cor18}
Let $i \in \{3, \ldots, n\}$. If $\gcd(r_b(i),a) = 1$, then $\omega \in \operatorname{Ap}(S_a(b,i))$ if and only if $\omega = b\, a^{(i)}_i$ or there exist $(u_2, \ldots u_{i-1}) \in R(b,i-1)$ and $u_i \in \{0, \ldots, b-1\}$ such that 
\begin{equation}\label{ecu3} 
\omega = \sum_{j=2}^{i-1} u_j a^{(i-1)}_j + b^{i-1} \left(\sum_{j=2}^{i-1} u_j \right) + u_i\, a^{(i)}_i.
\end{equation}
\end{corollary}

\begin{proof}
We observe that $a^{(i)}_j = a^{(i-1)}_j + r_b(i) - r_b(i-1) = a^{(i-1)}_j + b^{i-1},\ j = 0, \ldots, i-1$. So, \eqref{ecu3} becomes 
\[\omega = \sum_{j=1}^{i-1} u_j a^{(i)}_j + u_i\, a^{(i)}_i\] and, taking into account  \eqref{ecu2}, our claim readly follows from Theorem \ref{teo15}.
\end{proof}

Notice that, by Theorem \ref{teo15}, an immediate consequence of Corollary \ref{cor18} is that $\operatorname{Ap}(S_a(b,i))$ can be constructed from $\operatorname{Ap}(S_a(b,i-1))$ provided that both $S_a(b,i)$ and $S_a(b,i-1)$ are numerical semigroups; indeed, the first and second summands in the right hand side of \eqref{ecu3} correspond to the elements of $\operatorname{Ap}(S_a(b,i-1))$ and their lengths, respectively. Recall that, by Proposition \ref{prop17}, all elements in $\operatorname{Ap}(S_a(b,i-1))$ have the same length.

If $S$ is a homogeneous numerical semigroup, we write $\mathsf{m}_S(s)$ for the length of $s \in \operatorname{Ap}(S)$. The notation $\mathsf{m}_S(s)$ is usually reserved for the minimal length $s \in S$; clearly, no ambiguity occur in our case. 

The following result is an straightforward consequence of Corollary \ref{cor18}.

\begin{corollary}\label{cor19}
Let $i \in \{3, \ldots, n\}$. If $\gcd(r_b(i),a) = \gcd(r_b(i-1),a) = 1$, then $\omega \in \operatorname{Ap}(S_a(b,i))$ if and only if $\omega = b\, a^{(i)}_i$ or there exist $\omega' \in \operatorname{Ap}(S_a(b,i-1))$ and $u_i \in \{0, \ldots, b-1\}$ such that
\begin{equation}\label{ecu4} 
\omega = \omega' + b^{i-1}\, \mathsf{m}_{S_a(b,i-1)}(\omega') + u_i\, a^{(i)}_i.
\end{equation}
\end{corollary}

Corollary \ref{cor19} maintains a great similarity with \cite [Theorem 3.3]{OP18}; however, the techniques used are very different and, more importantly, we do not require the hypothesis $r_b(n) > a^2\, r_b(n-1)^2$ (see Remark \ref{rem9}).

\section{The Frobenius problem}\label{Sect4}

Let $a, b$ and $n$ be three positive integers such that $b > 1, n > 1$ and $\gcd(r_b(n),a)=1$. 

In this section we address the Frobenius problem for $S_a(b,n)$. More precisely, we will give a formula for $\operatorname{F}(S_a(b,n))$ in terms of $a,b$ and $n$. To do this we take advantage of the following result due to Selmer (see, e.g., \cite[Proposition 2.21]{libro}).

\begin{proposition}\label{prop20}
Let $S$ be a numerical semigroup and let $s \in S\setminus \{0\}$. Then,
\begin{enumerate}[(a)]
\item $\operatorname{F}(S) = \max \operatorname{Ap}(S,s) - s$;
\item $\operatorname{g}(S) = \frac{1}s \left(\sum_{\omega \in \mathrm{Ap}(S,s)} \omega \right) - \frac{s-1}2$.
\end{enumerate}
\end{proposition}

Before giving our formula for the Frobenius number, we show an interesting result which will be used below and later in the last section. As in the previous section we write $\operatorname{Ap}(S)$ for $\operatorname{Ap}(S,\operatorname{m}(S))$.

\begin{lemma}\label{lema21}
If $\alpha_i := a_i + \sum_{j=i}^n (b-1)\, a_j,\ i = 2, \ldots, n$, then 
the following holds:
\begin{enumerate}[(a)]
\item $\alpha_i \in \operatorname{Ap}(S_a(b,n))$, for every $i = 2, \ldots, n$
\item If $a < b^n-1,$ then $\alpha_2 > \ldots > \alpha_n$, and if $a > b^n-1$, then $\alpha_2 < \ldots < \alpha_n$.
\item For each $\omega \in \operatorname{Ap}(S_a(b,n)),$ there exists $i \in \{2, \ldots, n\}$ such that $\omega \leq \alpha_i$.
\end{enumerate}
\end{lemma}

\begin{proof}
Part (a) is nothing but a particular case of Theorem \ref{teo15}. To prove (b) it suffices to observe that \[\alpha_i - \alpha_{i+1} = b\, a_i - a_{i+1} = b^n-1-a,\ i = 2, \ldots, n-1,\] and note that $b^n-1 \neq a$ because $\gcd(r_b(n), a) = 1$ by hypothesis.
Finally, to prove (c) we can take advantage of Theorem \ref{teo15} which state that for each $\omega \in \operatorname{Ap}(S_a(b,n))$, there exist $(u_2, \ldots, u_n) \in R(b,n)$ such that $\omega =  \sum_{j=2}^n u_j\, a_j$. Clearly, by the definition of $R(b,n)$, if $u_i$ is the leftmost nonzero entry in $(u_2, \ldots, u_n) \in R(b,n)$, then \[\sum_{j=2}^n u_j\, a_j \leq a_i + \sum_{j=i}^n (b-1)\, a_j = \alpha_i\] and we are done. 
\end{proof}

\begin{theorem}\label{teo22}
The Frobenius number of $S_a(b,n)$ is equal to
\begin{enumerate}[(a)]
\item $(n-1)\,(b^n-1-a)+a\, a_1$, if $a < b^n-1;$ 
\item $b^n-1-a + a\, a_1$, if $a > b^n-1$.
\end{enumerate}
\end{theorem}

\begin{proof}
By Lemma \ref{lema21}, we have that $\max \operatorname{Ap}(S_a(b,n))$ is equal to either $ a_2 + \sum_{j=2}^n (b-1)\, a_j$, if $a < b^n-1$, or $b\, a_n$, if $a > b^n-1$. So, we distinguish two cases:
\begin{enumerate}[(a)]
\item If $a < b^n-1$, then $\max \operatorname{Ap}(S_a(b,n)) = \alpha_2$. So, by Selmer's formula (Proposition \ref{prop20}), we obtain that 
\begin{align*}
\operatorname{F}(S_a(b,n)) & =  \left(a_2 + \sum_{j=2}^n (b-1)\, a_j\right) - a_1 = \displaystyle{\sum_{j=2}^n (b-1)\, a_j} + a \\  & = 
\sum_{j=2}^n (b-1) \left(a_1 + a \sum_{k=0}^{j-2} b^j\right) + a \\ & =  (n-1)(b-1) a_1 + a \sum_{j=2}^n (b^{j-1}-1) +a \\ & = 
(n-1)(b-1) a_1 + a\, a_1 - (n-1) a \\ & =  (n-1)\,(b^n -1-a)+a\, a_1 .
\end{align*}
\item If $a > b^n-1$, then $\max \operatorname{Ap}(S_a(b,n)) = \alpha_n$. So, by Selmer's formula we conclude that \begin{align*}
\operatorname{F}(S_a(b,n)) & = b\, a_n - a_1 = b\, (a_1 + a\, r_b(n-1)) - a_1 \\ & =  
(b-1) a_1 + a\, b,\, r_b(n-1) = b^n - 1 + a\, (a_1 -1)\\ & = b^n-1-a + a\,a_1.  
\end{align*}
\end{enumerate}
\end{proof}

Observe that condition $a > b^n-1$ corresponds to the case considered in \cite{Ugo17} (see the comment after Corollary \ref{cor8}).

\section{The genus of grepunit semigroups}\label{Sect5}

In this section, we use Selmer's formulas (Proposition \ref{prop20}) to compute the genus of grepunit semigroups in terms of $a, b$ and $n$.

The following results state some useful properties of the sets $R(b,i)$. 

\begin{lemma}\label{lema23}
Let $b > 1$ be an integer. For each $i \geq 2$, the following holds:
\[\sum_{(u_2, \ldots u_i) \in R(b,i)} \left( \sum_{j=2}^i u_j \right) = \sum_{j=2}^i \frac{b^i+b^{i-(j-1)}}{2}.\]
\end{lemma}

\begin{proof}
We proceed by induction on $i$. If $i=2$, then $R(b,i) = \{0,1,\ldots,b\}$ and our claim readily follows. Suppose that $i > 2$ and that the result is true for $i-1$. Now, since by Lemma \ref{lema14} the cardinality of $R(b,i-1)$ is equal to $r_b(i-1)$, by \eqref{ecu2}, we have that \[\sum_{(u_2, \ldots u_i) \in R(b,i)} \left( \sum_{j=2}^i u_j \right) = \sum_{(u_2, \ldots u_{i-1}) \in R(b,i-1)} b\, \left( \sum_{j=2}^{i-1} u_j \right) + r_b(i-1) \left( \sum_{k=0}^{b-1} k \right) + b.\] So, by induction hypothesis, we obtain that 
\begin{align*}
\sum_{(u_2, \ldots u_i) \in R(b,i)} \left( \sum_{j=2}^i u_j \right) & = 
b\, \sum_{j=2}^{i-1} \frac{b^{i-1}+b^{(i-1)-(j-1)}}{2} + r_b(i-1) \left(\sum_{k=0}^{b-1} k\right) + b \\ & = \sum_{j=2}^{i-1} \frac{b^{i}+b^{i-(j-1)}}{2} + \frac{b^{i-1} - 1}{b-1}\, \frac{b (b-1)}2 + b \\ 
& =  \sum_{j=2}^{i-1} \frac{b^{i}+b^{i-(j-1)}}{2} + \frac{b^i + b}2 = \sum_{j=2}^i \frac{b^i+b^{i-(j-1)}}{2},
\end{align*}
as claimed. 
\end{proof}

As in previous sections, let $a, b$ and $n$ be three positive integers such that $b > 1$ and $n > 1$. Given $i \in \{2, \ldots, n\}$ we write \[a^{(i)}_1:=r_b(i), a^{(i)}_2 := a^{(i)}_1 + a\, r_b(1), \ldots, a^{(i)}_i :=  a^{(i)}_1 + a\, r_b(i-1).\]

\begin{proposition}\label{prop24}
Let $i \in \{2, \ldots, n\}.$ Then
\[\sum_{(u_2, \ldots u_i) \in R(b,i)} \left( \sum_{j=2}^i u_j a^{(i)}_j \right) = \sum_{j=2}^i \frac{b^i+b^{i-(j-1)}}{2}\, a^{(i)}_j.\]
\end{proposition}

\begin{proof}
We proceed by induction on $i$. If $i = 2$, then $R(b,i) = \{0, \ldots, b\}$ and \[\sum_{u_2 \in \{0, \ldots, b\}} \left( u_2\, a^{(2)}_2 \right) = \left( \sum_{u_2 \in \{0, \ldots, b\}}  u_2 \right)\, a^{(2)}_2 =  \frac{b\, (b+1)}2\, a_2^{(2)}.\] Suppose now $i > 2$ and that the result is true for $i-1$. Since $a_j^{(i)} = a_j^{(i-1)} + b^{i-1},\ j = 1, \ldots, i-1$, by \eqref{ecu2}, we have that
\begin{align*}
\sum_{(u_2, \ldots u_i) \in R(b,i)} \left( \sum_{j=2}^i u_j a^{(i)}_j \right) = & \sum_{(u_2, \ldots, u_{i-1}) \in R(b,i-1)} b \left( \sum_{j=2}^{i-1} u_j a^{(i-1)}_j \right) + \\ & + b^{i-1} \left(\sum_{(u_2, \ldots, u_{i-1}) \in R(b,i-1)} \left( \sum_{j=2}^{i-1} u_j \right) \right) + \\ & + r_b(i-1) \sum_{k=0}^{b-1} k\, a^{(i)}_i + b\, a_i^{(i)}.
\end{align*}
By induction hypothesis, the first summand of the right hand side is equal to \[b \left(\sum_{j=2}^{i-1} \frac{b^{i-1} + b^{i-1-(j-1)}}2\, a_j^{(i-1)} \right) = \sum_{j=2}^{i-1} \frac{b^{i} + b^{i-(j-1)}}2\, a_j^{(i-1)},\] and, by Lemma \ref{lema23}, the second summand of the right hand side is equal to 
\[b^{i-1}\, \left(\sum_{j=2}^{i-1} \frac{b^i + b^{i-(j-1)}}2\right) = \sum_{j=2}^{i-1} \frac{b^i + b^{i-(j-1)}}2\, b^{i-1} = \sum_{j=2}^{i-1} \frac{b^i + b^{i-(j-1)}}2\, \left(a_j^{(i)}-a_j^{(i-1)}\right).\] Therefore, 
\begin{align*}
\sum_{(u_2, \ldots u_i) \in R(b,i)} \left( \sum_{j=2}^i u_j a^{(i)}_j \right) & = \sum_{j=2}^{i-1} \frac{b^i + b^{i-(j-1)}}2\, a_j^{(i)} +  + r_b(i-1) \sum_{k=0}^{b-1} k\, a^{(i)}_i + b\, a_i^{(i)} \\ & = \sum_{j=2}^{i-1} \frac{b^i + b^{i-(j-1)}}2\, a_j^{(i)} +  \left(\sum_{k=0}^{i-2} b^k\right) \frac{(b-1)b}2\, a^{(i)}_i + b\, a^{(i)}_i \\ & = \sum_{j=2}^{i-1} \frac{b^i + b^{i-(j-1)}}2\, a_j^{(i)} + \frac{b^i + b}2\, a^{(i)}_i =  \sum_{j=2}^i \frac{b^i+b^{i-(j-1)}}{2}\, a^{(i)}_j,
\end{align*}
as claimed.
\end{proof}

The next result follows immediately from Theorem \ref{teo15} and Proposition \ref{prop24}.

\begin{corollary}\label{cor25}
Let $i \in \{2, \ldots, n\}.$  Then \[\sum_{\omega \in \operatorname{Ap}(S_a(b,i))} \omega = \sum_{j=2}^i \frac{b^i+b^{i-(j-1)}}{2}\, a_j^{(i)}.\] 
\end{corollary}

Now, as a direct consequence of Corollary \ref{cor25} and Selmer's formula (Proposition \ref{prop20}), we obtain the following formula for the genus of $S_a(b,n)$, provided that $\gcd((r_b(n), a) = 1$.
\begin{equation}\label{ecu5}
\begin{split}
\operatorname{g}(S_a(b,n)) & = \frac{\sum_{j=2}^n \displaystyle{\frac{b^n+b^{n-(j-1)}}{2}}\, a_j}{a_1} - \frac{a_1-1}{2} \\ & = \frac{1}{a_1}\, \sum_{j=2}^n \frac{b^n+b^{n-(j-1)}}{2} \, a_j - \frac{a_1-1}{2} 
\end{split}
\end{equation}
where $a_j = a_j^{(n)},\ j = 1, \ldots, n$, as usual.

\begin{corollary}\label{cor26}
The genus of $S_a(b,n)$ is equal to \[\operatorname{g}(S_a(b,n)) =  \frac{(n-1)\, b^n + (a_1-1)\, a}2.\]
\end{corollary}

\begin{proof}
Since $a_j = a_1 + a\, r_b(j-1)$, by equation \eqref{ecu5}, we have that 
\begin{align*}
\operatorname{g}(S_a(b,n)) = & \frac{1}{2\, a_1}\, \sum_{j=2}^n (b^n+b^{n-(j-1)}) \, (a_1 + a\, r_b(j-1))  - \frac{a_1-1}{2} \\ = & \frac{1}2 \sum_{j=2}^n (b^n + b^{n-(j-1)}) + \frac{a}{2 a_1} \sum_{j=2}^n (b^n + b^{n-(j-1)})\, r_b(j-1) - \frac{a_1-1}2 \\ = & \frac{(n-1)\, b^n}2 + \frac{a_1-1}2 + \frac{a}{2 a_1} \sum_{j=2}^n (b^n + b^{n-(j-1)})\, r_b(j-1) - \frac{a_1-1}2 \\ = & \frac{(n-1)\, b^n}2 + \frac{a}{2 a_1} \sum_{j=2}^n (b^n + b^{n-(j-1)})\, r_b(j-1).
\end{align*}
Now, taking into account that $r_b(j-1) = \frac{b^{j-1} - 1}{b-1}$, we obtain that
\begin{align*}
\operatorname{g}(S_a(b,n)) = & \frac{(n-1)\, b^n}2 + \frac{a}{2 a_1 (b-1)} \sum_{j=2}^n (b^n + b^{n-(j-1)})\, (b^{j-1} - 1) \\ = &  \frac{(n-1)\, b^n}2 + \frac{a}{2 a_1 (b-1)} \sum_{j=1}^{n-1} (b^{n+j} - b^{n-j}) \\ = &  \frac{(n-1)\, b^n}2 + \frac{a}{2 a_1 (b-1)} \Big( (b^n-1) \sum_{j=1}^{n-1} b^j \Big)  =\frac{(n-1)\, b^n + (a_1-1)\, a}2, 
\end{align*}
as claimed.
\end{proof}

\begin{example}\label{exa27}
If $a = b = 3$ and $n = 4$, then the grepunit semigroup $S = S_a(b,n)$ is minimally generated by $a_1 = 40, a_2 = 43, a_3 = 52$ and $a_4 = 79$. By Corollary \ref{cor25}, we have that $\sum_{\omega \in \operatorname{Ap}(S)} \omega = 54\, a_2 + 45\, a_3 + 42\, a_4 = 7980.$ So, by \eqref{ecu5}, we conclude that \[\operatorname{g}(S) = \frac{7980}{40} - \frac{39}{2} = 180.\] Note that, by Corollary \ref{cor26}, we can get $\operatorname{g}(S)$ without computing $\sum_{\omega \in \operatorname{Ap}(S)} \omega$.
\end{example}

\section{The type of grepunit semigroups. Wilf's conjecture}\label{Sect6}

Let $S$ be a numerical semigroup and let $\operatorname{n}(S)$ be the cardinality of $\{s \in S \mid s < \operatorname{F}(S)\}$. Clearly, $\operatorname{g}(s) + \operatorname{n}(S) = \operatorname{F}(S)+1$. In \cite{wilf}, {H.S.} Wilf conjectured that \begin{equation}\label{ecu6}\operatorname{F}(S) \leq \operatorname{e}(S)\, \operatorname{n}(S) - 1,\end{equation} where $\operatorname{e}(S)$ denotes the embedding dimension of $S$. 

Although there are many families of numerical semigroups for which this conjecture is known to be true, the general case remains unsolved. The numerical semigroups that satisfy Wilf's conjecture are said to be Wilf (see for example the survey \cite{Delgado}).

In this section, we will prove that the grepunit semigroups are Wilf. Since $\operatorname{n}(S) = \operatorname{F}(S)-\operatorname{g}(s)+1$, we can take advantage of our explicit formulas for the Frobenius number (Theorem \ref{teo22}) and for the genus (Corollary \ref{cor26}) to check that \eqref{ecu6} holds. However, we will follow a different approach: we will prove that the grepunit semigroups are Wilf as an immediate consequence of the computation of their pseudo-Frobenius numbers.

Recall that an integer $x$ is a pseudo-Frobenius number of a numerical semigroup $S$, if $x \not\in S$ and $x+S \in S$, for all $s \in S \setminus \{0\}$. We denote by $\operatorname{PF}(S)$ the set of pseudo-Frobenius numbers of $S$. The cardinality of $\operatorname{PF}(S)$ is called the type of $S$ and it is denoted $\operatorname{t}(S)$.

Given a numerical semigroup $S$, we write $\preceq_S$ for the partial order on $\mathbb{Z}$ such that $y \preceq_S x$ if and only if $x - y \in S$.

The following result is \cite[Proposition 2.20]{libro}.

\begin{proposition}\label{prop28}
Let $S$ be a numerical semigroup. If $s$ is a nonzero element of $S$, then
\[\operatorname{PF}(S) = \left\{\omega - s\ \mid\ \omega \in \operatorname{Maximals}_{\preceq S}\, \operatorname{Ap}\big(S,s\big) \right\}.\]
\end{proposition}

As in the previous sections, let $a, b$ and $n$ be three positive integers such that $b > 1, n > 1$ and $\gcd(r_b(n),a) = 1$.

\begin{proposition}\label{prop29}
The set of maximal elements of $\operatorname{Ap}\big(S_a(b,n)\big)$ with respect to $\preceq_{S_a(b,n)}$ is equal to  \[\Big\{a_i + (b-1) \sum_{j=i}^n a_j\ \mid\ i = 2, \ldots, n\Big\}.\]
\end{proposition}

\begin{proof}
Set $S:=S_a(b,n)$ and $\alpha_i := a_i + \sum_{j=i}^n (b-1)\, a_j,\ i = 2, \ldots, n$. 
From the proof of Lemma \ref{lema21}(c), we have that, for each $\omega \in \operatorname{Ap}(S)$, there exists $\alpha_i$ such that $\alpha_i - \omega \in S$. Therefore $\operatorname{Maximals}_{\preceq S} \operatorname{Ap}(S) \subseteq \{\alpha_2, \ldots, \alpha_n\}$. Now, since by Lemma \ref{lema21}, $\alpha_2 > \ldots > \alpha_n,$ if $a < b^n-1$, and $\alpha_2 < \ldots < \alpha_n,$ if $a > b^n-1$; in order to prove the reverse inclusion, we distinguish two cases:
\begin{itemize}
\item If $a < b^n-1$, then $\alpha_j - \alpha_i = (i-j) (b^n-(a+1)) = \alpha_{n-(i-j)} - \alpha_n$, for every $i > j$. Now, if there exist $i > j$ such that $\alpha_i \preceq_S \alpha_j$, then $\alpha_{n-(i-j)} - \alpha_n \in S$. So, $\alpha_{n-(i-j)} = \alpha_n + \sum_{l=2}^n u_l a_l$ for some $u_l \in \mathbb{N},\ l = 2, \ldots, n$, not all zero. If $k \in \{2, \ldots, n\}$ is such that $u_k \neq 0$, then $\alpha_{n-(i-j)} = \alpha_n + a_k + s$, where $s = \sum_{l=2}^n u_l a_l-a_k \in S$. Thus, since $\alpha_n = b a_n$ and $b a_n + a_k = b a_{k-1} + (a+1) a_1$, we conclude that $\alpha_{n-(i-j)} - a_1 = b a_{k-1} + a\, a_1 + s \in S$, which is not possible because $\alpha_{n-(i-j)} \in \operatorname{Ap}(S)$. 
\item If $a > b^n-1$, then $\alpha_j - \alpha_i = (j-i) ((a+1)-b^n) = \alpha_{j-i+2} - \alpha_2,$ for every $i < j$. Now, if there exist $i > j$ such that $\alpha_i \preceq_S \alpha_j$, then $\alpha_{j-i+2} - \alpha_2 \in S$.  So, $\alpha_{j-i+2} = \alpha_2 + \sum_{l=2}^n u_l a_l = a_2 + \sum_{l=2}^n (u_l + b-1) a_l$, for some $u_l \in \mathbb{N},\ l = 1, \ldots, n$, not all zero. If $k \in \{2, \ldots, n\}$ is such that $u_k \neq 0$, then $\alpha_{n-(i-j)} = a_2 + b\, a_k + s$, where $s = \sum_{l=2}^n (u_l + b-1) a_l - b\, a_k \in S$.  Thus, since $a_2 + b\, a_k = b a_1 + a_{k+1}, $ if $k < n$, and $a_2 + b\, a_k = (b+a+1)\, a_1$, if $k = n$, we conclude that $\alpha_{j-i+2} - a_1 \in S$, a contradiction again.
\end{itemize}
Finally, since we have shown that $\alpha_i - \alpha_j \not\in S_a(b,n)$, for every $i,j \in \{2, \ldots, n\}$, we conclude that $\alpha_i \in \operatorname{Maximals}_{\preceq S}(\operatorname{Ap}(S))$, for every $i = 2, \ldots, n$, as desired. 
\end{proof}

\begin{corollary}\label{cor30}
The set of pseudo-Frobenius numbers of $S_a(b,n)$ is equal to \[\left\{(n-i+1)\,(b^n-1-a) + a\, a_1 \ \mid\ i = 2, \ldots, n \right\}.\] Consequently, the type of $S_a(b,n)$ is $n-1$.
\end{corollary}

\begin{proof}
By Propositions \ref{prop28} and \ref{prop29},
$\operatorname{PF}(S_a(b,n) = \{a_i + (b-1)\sum_{j=i}^n a_j-a_1\ \mid\ i = 2, \ldots, n \}.$ Now, taking into account that 
\begin{align*}
a_i + (b-1)\sum_{j=i}^n a_j-a_1 & = a\, r_b(i-1) + (b-1) (n-i+1) a_1 + (b-1)\, a\, \sum_{j=i}^n r_b(j-1) \\ & =  a\, r_b(i-1) + (n-i+1)(b^n-1) + a\, \sum_{j=i}^n (b^{j-1}-1) \\ & = a\, r_b(i-1) + (n-i+1)(b^n-1-a) + a\, ( a_1 - r_b(i-1)) \\ & =  (n-i+1)(b^n-1-a) + a\, a_1, 
\end{align*}
for every $i \in \{2, \ldots, n\}$, and we are done.
\end{proof}

Finally, since, by \cite[Theorem 20]{FGH}, $\operatorname{F}(S) \leq (\operatorname{t}(S)+1)\, \operatorname{n}(S)-1$, for every numerical semigroup $S$, and, by Corollary \ref{cor30}, $\operatorname{t}(S_a(b,n)) + 1 = \operatorname{e}(S_a(b,n))$, we conclude that generalized repunit numerical semigroups are Wilf.  

\medskip
\noindent\textbf{Acknowledgments.}{The authors wish to thank the referees for their valuable comments and suggestions that have undoubtedly helped to increase the clarity of the presentation of this paper.}

\medskip
\noindent\textbf{Data availability statement.}{All data generated or analysed during this study are included in this published article.}



\begin{thebibliography}{99}

\bibitem{BCO2}
{Branco, M.B.; Cola\c{c}o, I., Ojeda, I.} {Minimal binomial systems of generators for the ideals of certain monomial curves}. Mathematics \textbf{9}(24), 2021, 3204. 

\bibitem{Delgado}
{Delgado, M.}: \emph{Conjecture of Wilf: a survey}. Numerical semigroups, 39--62, Springer INdAM Ser., \textbf{40}, Springer, Cham, 2020.

\bibitem{FGH}
{Fr\"oberg, R.; Gottlieb, C.; H\"aggkvist R.}: \emph{On numerical semigroups}, Semigroup Forum \textbf{35} (1987), 63--83.

\bibitem{GSOj}
{Garc\'{\i}a-S\'anchez, P.A.; Ojeda, I.}: \emph{Uniquely presented finitely generated commutative monoids}. Pacific J. Math. \textbf{248}(1), (2010), 91--105. 

\bibitem{JZ18}
{Jafari, R., Zarzuela Armengou, S.}: \emph{Homogeneous numerical semigroups}, Semigroup Forum \textbf{97}, (2018), 278--306.

\bibitem{MS05}
{Miller, E., Sturmfels, B.}: \emph{Combinatorial commutative algebra}, vol. 227 of Graduate Texts in Mathematics. Springer-Verlag, New York, 2005.

\bibitem{OjVi2}
{Ojeda, I; Vigneron-Tenorio, A.} \emph{Indispensable binomials in semigroups ideals}, Proc. Amer. Math. Soc. \textbf{138} (2010), 4205--4216

\bibitem{OP18}
{O'Neill, C.; Pelayo, R.}: \emph{Ap\'ery sets of shifted numerical monoids}, Adv. Appl. Math. \textbf{97} (2018), 27--35.

\bibitem{ramirez-alfonsin}
{Ram\'{\i}rez Alfons\'{\i}n, {J.L.}}: \emph{The Diophantine Frobenius Problem}, Oxford Lecture Series in Mathematics and Its Applications, 2005.

\bibitem{libro}  
{{Rosales, J.C.}; {Garc\'{\i}a-S\'{a}nchez, P.A.}}: \emph{Numerical semigroups}. Developments in Mathematics, vol.\textbf{20}, Springer, New York, (2009).

\bibitem{RBT16} 
{{Rosales, J.C.}; {Branco, M.B.}; {Torr\~ao, D.}}: \emph{The Frobenius problem for repunit numerical semigroups}, Ramanujan J. \textbf{40} (2016), 323--334.

\bibitem{RBT17} 
{{Rosales, J.C.}; {Branco, M.B.}; {Torr\~ao, D.}}: \emph{The Frobenius problem for Mersenne numerical semigroups}, Math. Z. \textbf{286} (2017), 741--749.

\bibitem{Ugo17}
{Ugolini, S.}: \emph{On numerical semigroups closed with respect to the action of affine maps}, Publ. Math. Debrecen \textbf{90} (2017), 149--167.

\bibitem{wilf}
{Wilf, H.S.}: \emph{A circle-of-lights algorithm for the ``money-changing problem''}. Am. Math. Mon. \textbf{85}(7), 1978, 562--565.
\end{thebibliography}
\end{document}